\documentclass[a4paper,11pt]{amsart}
\usepackage{geometry}
\usepackage[bookmarksopen=false,pdftex=true,breaklinks=true,%
      backref=page,pagebackref=true,plainpages=false,%
      hyperindex=true,pdfstartview=FitH,colorlinks=true,%
      pdfpagelabels=true,colorlinks=true,linkcolor=blue,%
      citecolor=red,urlcolor=green,hypertexnames=false%
      ]%
   {hyperref}
  
 \usepackage{amsfonts,mathrsfs}
 \usepackage{amsthm}
 \usepackage{amssymb}
 \usepackage{enumerate}
 \usepackage{tikz-cd}
 \usepackage{cleveref}

\theoremstyle{definition}
\newtheorem{Def}{Definition}[section]
\newtheorem{ex}[Def]{Example}
\newtheorem{rem}[Def]{Remark}

\theoremstyle{plain}

\newtheorem{thm}[Def]{Theorem}
\newtheorem*{thm*}{Theorem}
\newtheorem{lem}[Def]{Lemma}
\newtheorem{cor}[Def]{Corollary}
\newtheorem*{cor*}{Corollary}
\newtheorem{con}[Def]{Conjecture}
\newtheorem*{con*}{Conjecture}
\newtheorem*{frag*}{Question}
\newtheorem*{verm*}{Vermutung}

\usepackage{mathrsfs}

\newcommand{\Seq}{\operatorname{Seq}}
\newcommand{\sq}{\operatorname{sq}}
\newcommand{\pr}{\operatorname{pr}}

\newcommand{\Sym}{\operatorname{Sym}}

\newcommand{\Hom}{\operatorname{Hom}} 
\newcommand{\Spec}{\operatorname{Spec}}

\newcommand{\cI}{{\mathcal I}}

\newcommand{\fP}{{\mathfrak P}}

\newcommand{\fm}{{\mathfrak m}}

\newcommand{\fp}{{\mathfrak p}}

\newcommand{\A}{{\mathbb A}}
\newcommand{\G}{{\mathbb G}}
\newcommand{\C}{{\mathbb C}}
\newcommand{\R}{{\mathbb R}}
\newcommand{\pp}{\mathbb{P}}

\newcommand{\N}{{\mathbb N}}

\title[Equivariant geometry in infinite affine space]{Equivariant algebraic and semi-algebraic geometry of infinite affine space}
\author{Mario Kummer}
\address{Technische Universit\"at Dresden, Germany} 
\email{mario.kummer@tu-dresden.de}
\author{Cordian Riener}
\address{UiT The Arctic University of Norway}
\email{cordian.riener@uit.no}

\newcommand{\comment}[1]{}

\begin{document}


\begin{abstract}
 We study $\Sym(\infty)$-orbit closures of not necessarily closed points in the Zariski spectrum of the infinite polynomial ring $\C[x_{ij}:\, i\in\N,\,j\in[n]]$. Among others, we characterize invariant prime ideals in this ring. Furthermore, we study projections of basic equivariant semi-algebraic sets defined by  $\Sym(\infty)$ orbits of polynomials in $\R[x_{ij}:\, i\in\N,\,j\in[n]]$. For $n=1$ we prove a quantifier elimination type result which fails for $n>1$. 
\end{abstract}

\maketitle

\section{Introduction}
From the point of view of commutative algebra the ring $\C[x_{ij}:\, i\in\N,\,j\in[n]]$, where $[n]=\{1,\ldots, n\}$, is a complicated object, as it is, for example, not  a Noetherian. On the other hand, various results from classical commutative algebra from finite dimensional commutative rings have been established in an equivariant version, with respect to the natural permutation action of the infinite symmetric group $\Sym(\infty)$ on the first index.  Most notably, it was observed by Cohn in  \cite{Co67, cohen} that symmetric ideals in this ring satisfy the ascending chain condition, i.e., it is $\Sym(\infty)$-Noetherian. This result was independently recovered by  Aschenbrenner and Hillar \cite{AH07}. The results presented in this article are rather of geometrical flavor. Compared to the sophisticated set of techniques (see for example \cite{soph1,soph3,soph2}) that has been developed by various authors in the subsequent years we will pursue an elementary approach to understand the geometry of symmetric algebraic and semi-algebraic sets in the spectrum $\A^n_\infty$ of this ring. In fact the proofs of our main results do not even make use of equivariant Noetherianity.

The essence of our approach is to study orbit closures of (not necessarily) closed points of $\A^n_\infty$. After characterizing orbit closures of closed points in \Cref{prop:orbitclos}, we prove in \Cref{thm:symprime} that fixed points of $\A^n_\infty$, i.e. invariant prime ideals, are exactly those points that arise as the image of some natural fixed points under the base change of $\A^n_\infty$ to some field extension. In \Cref{thm:familiesfixed} we give an alternative description of (not necessarily closed) fixed points. From this one can easily construct infinite specialization chains of fixed points in $\A^n_\infty$ when $n>1$ (\Cref{ex:infichain}). Note that it is well-known that this cannot happen when $n=1$. In this case any such chain has length $\leq3$, see also \Cref{ex:n1}. This highlights the special role played by the $n=1$ case which we also observe in \Cref{sec:semialg} in the semi-algebraic setup.  While in the case $n=1$ our results are not as strong as the results from \cite{nagpal2020symmetric, nagpal2021symmetric} where a characterization of invariant radical ideals \cite{nagpal2020symmetric} and invariant ideals \cite{nagpal2021symmetric} is given, our results also apply to larger $n>1$. Our results lead naturally to \Cref{con:sober} which states that the Kolmogorov quotient of the orbit space is a spectral space, i.e. it is homeomorphic to the Zariski spectrum of a commutative ring \cite{hochster}. Together with \Cref{thm:orbitclos} this would provide a description of invariant radical ideals.

We further try the same approach of studying orbit closures for the semi-algebraic setup in order to establish an equivariant version of quantifier elimination. Again the situation with $n>1$ differs substantially from the case $n=1$. In \Cref{thm:tarskiseid} we prove an equivariant version of Tarski--Seidenberg for the case $n=1$ and \Cref{ex:last} shows that the analogous statement fails for $n>1$.
This fits with \cite[Theorem on p.~14]{moitra} where it is shown  that a naive version of quantifier elimination over the reals cannot exist in the infinite symmetric setup. Note that for the setup considered  in \cite{moitra}, where both index sets are infinite, also the results on the side of commutative algebra are generally less well behaved \cite[Proposition 5.2]{AH07}.  

\hspace{0.4cm}

 \noindent \textbf{Acknowledgements.}
This article was initiated during the visit of the first author to Tromsø in 2018 financed through the \emph{Pure Mathematics in Norway} project of the Trond Mohn foundation. The second author also acknowledges financial support by the Tromsø Research foundation under grant 17mattecr. We further would like to thank Jan Draisma for helpful comments following a talk on the results of this article by the first author at the SIAM Conference on Applied Algebraic Geometry 2019 in Bern. 

\section{Preliminaries from algebraic geometry}
In the following let $K$ always denote an algebraically closed field with uncountably many elements. The affine scheme $\Spec(K[x_1,\ldots,x_n])$ over $K$ is denoted by $\A^n_K$ or by $\A^n$ if the ground field does not matter or is clear from the context. We are mostly interested in the case $K=\C$ but in the course of the proofs we have to work over other algebraically closed fields. Consider a (not necessarily finitely generated) $K$-algebra $A$ and let $X=\Spec(A)$ the associated affine scheme over $K$. We denote by $X^{\textnormal{max}}$ the subspace of $X$ consisting of all closed points of $X$, i.e. maximal ideals of $A$.
Recall that for any field extension $K\subset L$ the set $X(L)=\Hom_K(A,L)$ of $L$-rational points admits a natural map $X(L)\to X$ that sends a homomorphism of $K$-algebras to its kernel. If $K=L$, this map is injective. The pullback of the Zariski topology on $X$ to $X(L)$ is called the Zariski topology on $X(L)$. The base change $X_L$ of $X$ to $L$ is the affine scheme $\Spec(A\otimes_K L)$ over $L$. Composing with the natural map $A\to A\otimes_K L$ gives a natural map $X_L(L)\to X(L)$ where $X_L(L)=\Hom_L(A\otimes_K L,L)$. By the universal property of the tensor product $\otimes_K$  this map is a bijection. Further, the map $X_L(L)\to X(L)$ is continuous but in general not a homeomorphism.

In analogy to the definition of a scheme of finite type over $K$, we define an \emph{affine scheme of countable type over $K$} to be a scheme over $K$ of the form $\Spec(A)$ where $A$ can be generated by countably many elements as an $K$-algebra. Clearly, any closed subscheme of an affine scheme of countable type over $K$ is again an affine scheme of countable type over $K$. Recall the following result by Lang \cite{langnullstellensatz}.

\begin{thm}[Lang]\label{thm:langcount}
 Let $\cI$ be some index set of countable cardinality and consider $B=K[x_i:\, i\in \cI]$ the polynomial ring with variables indexed by $\cI$. Further let $X=\Spec(B)$, $J\subset B$  an ideal and $Z\subset X(K)$ the set of $K$-rational points of $X$ where all elements of $J$ vanish. Then the following holds:
 \begin{enumerate}
     \item If $f\in B$ vanishes on $Z$, then $f^k\in J$ for some $k\in\N$.
     \item If $J\neq B$, then $Z$ is not empty.
     \item Let $K\subset L$ be some field extension and $a_i\in L$ for $i\in\cI$.  The ring extension  $K[a_i:\,i\in\cI]$ is a field if and only if $a_i\in K$ for all $i\in\cI$.
 \end{enumerate}
\end{thm}

From this we can deduce that affine schemes of countable type over $K$ share several desirable properties with affine schemes of finite type over $K$.

\begin{cor}\label{lem:langcount}
 Let $X=\Spec(A)$ be an affine scheme of countable type over $K$.
 \begin{enumerate}
     \item $X(K)$ is dense in $X$.
     \item The image of $X(K)\to X$ is $X^{\textnormal{max}}$.
     \item If $I\subset A$ is an ideal and $Z\subset X(K)$ the zero set of $I$, then the set of all $f\in A$ that vanish on $Z$ is the radical ideal $\sqrt{I}$.
 \end{enumerate}
\end{cor}

\begin{proof}
 By assumption there is a countable index set $\cI$ and a surjective homomorphism $\varphi:B\to A$ of $K$-algebras where $B=K[x_i:\, i\in \cI]$. Let $J=\ker(\varphi)$.
 In order to prove the first statement, let $f\in A$ vanish on $X(K)$. If $g\in B$ is some element with $\varphi(g)=f$, then by (1) of \Cref{thm:langcount} we have $g^k\in J$ for some $k\in\N$ which shows that $f$ is nilpotent in $A$ and thus vanishes on $X$. This shows that $X(K)$ is dense in $X$. 
 For the second statement let $\fm$ be a maximal ideal of $A$. Then $A/\fm$ is a field extension of $K$ which is generated as $K$-algebra by the residue classes of $\varphi(x_i)$ for $i\in\cI$. Thus $A/\fm$ must be $K$ by (3) in \Cref{thm:langcount}. On the other hand, the kernel of any homomorphism $A\to K$ of $K$-algebras is a maximal ideal.
 
 For the third statement let $g\in B$ such that $\varphi(g)=f$. Then $g$ vanishes on the set of all $K$-rational points of $\Spec(B)$ where all elements of $\varphi^{-1}(I)$ vanish. Thus by (1) in \Cref{thm:langcount} we have that $g$ is in the radical of $\varphi^{-1}(I)$ which implies that $f$ is in the radical of $I$.
\end{proof}

\begin{cor}
 Let $f:X\to Y$ be a dominant morphism of affine schemes of countable type over $K$. The image of $X(K)$ is both dense in $Y(K)$ and in $Y$.
\end{cor}

For any field extension $K\subset L$ and any two affine schemes $X,Y$ over $K$ we have $(X\times Y)(L)=X(L)\times Y(L)$ (as sets).  We prove some basic facts on such products.

\begin{lem}\label{lem:opendensedense}
 Let $T$ be a topological space and $S\subset T$ a dense subset. Then $S\cap U$ is dense for every open dense subset $U\subset T$.
\end{lem}

\begin{proof}
 Let $V\subset T$ be an open subset with $S\cap U\cap V=\emptyset$. Then the closure of $S$ is contained in the complement of $U\cap V$. Thus $U\cap V=\emptyset$ since $S$ is dense. But then by the same argument $V=\emptyset$ since $U$ is dense as well.
\end{proof}

\begin{lem}\label{lem:proddense}
 Let $X$ and $Y$ be affine schemes over $K$. Let $A\subset X(K)$ and $B\subset Y(K)$ be dense subsets of $X(K)$ and $Y(K)$ respectively. Then $A\times B$ is dense in $(X\times Y)(K)=X(K)\times Y(K)$.
\end{lem}

\begin{proof}
 For every $a\in A$ the set $\{a\}\times B$ is dense in $\{a\}\times Y(K)$. Thus the closure of $A\times B$ contains $A\times Y(K)$. The closure of $A\times Y(K)$ contains $X(K)\times Y(K)$ by the same argument. Thus the closure of $A\times B$ is $X(K)\times Y(K)$.
\end{proof}

\begin{rem}
 If $K\subset L$ is a field extension, the corresponding statement of \Cref{lem:proddense} for dense subsets $A\subset X(L)$ and $B\subset Y(L)$ fails to be true in general. Indeed, let $L=K(t)$ the rational function field, $X=Y=\A^1_K$ and $A=B=\{t\}$. Then $A$ and $B$ are dense in $\A^1_K$ but $A\times B$ is contained in the diagonal $\Delta=V(x_1-x_2)\subsetneq\A^2_K$.
\end{rem}

\begin{cor}\label{cor:proddense}
 Let $X$ be an affine scheme over $K$ and $Y\subset X(K)$ a dense subset of $X(K)$. Then $Y^m=Y\times\cdots\times Y$ is dense in $(X^m)(K)=X(K)\times\cdots\times X(K)$.
\end{cor}

\begin{proof}
 We prove the claim by induction on the number $m$ of factors. The claim is clear for $m=1$. The induction step follows from \Cref{lem:proddense}.
\end{proof}

\begin{lem}\label{lem:diagonalthin}
 Let $X$ be an affine scheme of finite type over $K$. Let $X=X_1\cup\cdots\cup X_s$ be the decomposition of $X$ into irreducible components. Furthermore, assume that $X_i(K)$ is infinite for all $i=1,\ldots, s$. The subset $$B=\{(x_{1},\ldots,x_r)\in X(K)^{r}:\, x_i\neq x_j \textrm{ for all }i\neq j\}$$is dense in $X(K)^{r}$.
\end{lem}

\begin{proof}
  Since $K$ is algebraically closed, we have that the $K$-rational points of the irreducible components of $X^r$ are the sets of the form $X_{i_1}(K)\times\cdots\times X_{i_r}(K)$ for $i_j\in[s]$. None of these is contained in one of the closed subsets of the form $$\Delta_{ij}=\{(x_{1},\ldots,x_r)\in X(K)^{r}:\, x_i= x_j\}$$ for $i\neq j$ since each $X_i(K)$ is infinite. This implies the claim.
\end{proof}

\section{Orbit closures in $\A^n_{\infty}$}
We let $S=K[x_{ij}:\, i\in\N,\,j\in[n]]$. There is a natural action of the infinite symmetric group   $G=\textrm{Sym}(\infty)=\cup_{k\in\N}\textrm{Sym}(k)$ on $S$ given by $\sigma(x_{ij})=x_{\sigma(i)j}$ for all $\sigma\in G$ and $i\in\N$, $j\in[n]$. We denote by $\A^n_{K,\infty}$ the integral affine scheme $\Spec(S)$ of countable type over $K$. If the ground field $K$ is clear from the context, we will just write $\A^n_{\infty}=\A^n_{K,\infty}$. Throughout the article $L$ will always denote a field extension of $K$. The action of $G$ on $S$ induces an action on $\A^n_{\infty}$ and on $\A^n_\infty(L)$ for every field extension $L$ of $K$. We identify $\A^n_{\infty}(L)$ with the set of sequences $(p_i)_{i\in\N}\subset\A^n(L)= L^n$. Our goal is to characterize the closures of orbits of single points in $\A^n_{\infty}$: 
\begin{Def}
Given any $x\in\A^n_\infty$, the \emph{orbit closure} $\overline{Gx}$ of $x$ is the Zariski closure of the orbit $$Gx=\{\sigma(x):\, \sigma\in G\}.$$ If we equip $\overline{Gx}$ with the reduced induced closed subscheme structure, then $\overline{Gx}$ is a reduced affine scheme of countable type over $K$ on which the group $G$ acts. The orbit closure of an $L$-rational point is defined to be the orbit closure of the associated point in $\A^n_\infty$.
\end{Def}

\subsection{Orbit closures of closed points in $\A^n_\infty$} We first compute the orbit closures of $K$-rational points of $\A^n_\infty$ which are exactly the closed points of $\A^n_\infty$ by \Cref{lem:langcount}. For this we need to set up some notation.

\begin{Def}
 Let $p=(p_i)_{i\in\N}\subset \A_K^n(L)$ be a point from $\A^n_{\infty}(L)$. We denote by $V_p\subset\A^n_K$ the closure of the set $\{p_i:\,i\in\N\}$. Further for any $M\subset \A^n_K(L)$ we let $\nu(p,M)=|\{i\in\N:\,p_i\in M\}|$. For $q\in \A^n(L)$ we write $\nu(p,q):=\nu(p,\{q\})$.
\end{Def}

\begin{lem}\label{lem:vpcont}
 Let $p\in\A^n_{\infty}(L)$ and $X\subset\A^n_\infty$ its orbit closure. For every $q\in X(L)$ we have $V_q\subset V_p$.
\end{lem}

\begin{proof}
 Let $p=(p_i)_{i\in\N}\subset \A_K^n(L)$ and $q=(q_i)_{i\in\N}\subset \A_K^n(L)$.
 Let $f\in K[x_1,\ldots,x_n]$ be a polynomial that vanishes on $V_p$. This implies that $p$ is an $L$-rational point of the closed subscheme $Y$ of $\A^n_\infty$ defined by the ideal $I\subset S$ generated by all $g_i=f(x_{i1},\ldots,x_{in})$ for $i\in\N$. The subscheme $Y$ is closed, invariant under $G$ and $Y(L)$ contains $p$. Thus we have $X\subset Y$ and in particular we have $q\in Y(L)$. This shows that $f(q_i)=0$ for all $i\in\N$ and thus implies the claim.
\end{proof}

\begin{lem}\label{lem:orbitclossubset}
 Let $p\in \A^n_{\infty}(L)$. Consider the decomposition $$V_p=V_1\cup\cdots\cup V_r$$ into irreducible components and let $W_k=V_k(L)\setminus (\cup_{i\neq k} V_i(L))$. Let $X$ be the orbit closure of $p$. Then $X(L)$ is contained in $$\mathcal{W}_p:=\{w=(w_i)_{i\in\N}\subset V_p(L):\,\forall k\in[r]:\, \nu(w,W_k)\leq \nu(p,W_k)\}.$$ 
\end{lem}

\begin{proof}
 For every $w=(w_i)_{i\in\N}\in X(L)$ we must have $w_i\in V_p(L)$ for all $i\in\N$ by \Cref{lem:vpcont}. For all $i\in\N$ and $k=1,\ldots,r$ we let $I_{ik}\subset K[x_{i1},\ldots,x_{in}]$ be the vanishing ideal of $\cup_{j\neq k} V_j$. Note that the first index of $I_{ik}$ only indicates the labeling of the variables. If $m=\nu(p,W_k)$ is finite, then all elements of the ideal generated by $$\sigma(f) \textrm{ where }f\in I_{1k}\cdots I_{m+1,k}, \, \sigma\in G$$ vanish on $p$ and thus on $X$. That implies that $\nu(w,W_k)\leq m$ for all $w\in X(L)$ and thus $X(L)\subset \mathcal{W}_p$.
\end{proof}

Now we are able to describe the orbit closure of any closed point. By \Cref{lem:langcount} any closed subset of $\A^n_\infty$ can be understood in terms of its $K$-rational points.

\begin{thm}\label{prop:orbitclos}
 Let $p\in\A^n_{\infty}(K)$ and assume that the decomposition of $V_p$ into irreducible components has the form $$V_p=V_1\cup\cdots\cup V_r\cup\{v_1\}\cup\ldots\cup\{v_s\}$$ where the $V_i(K)$ are infinite and the $v_i$ are closed points of $\A^n_K$. Let $X$ be the orbit closure of $p$. The set $X(K)$ of $K$-rational points of $X$ consists exactly of those sequences in $V_p(K)$ where each $v_j$ appears at most as often as in $p$.
\end{thm}

\begin{proof}
 The claim is equivalent to $X(K)=\mathcal{W}_p$ defined in \Cref{lem:orbitclossubset} and $X(K)\subset \mathcal{W}_p$ is the statement of \Cref{lem:orbitclossubset} (for $K=L$).
 
 Let $w\in \mathcal{W}_p$, this amounts to $w=(w_i)_{i\in\N}\subset V_p(K)$ such that for all $1\leq j\leq s$ we have $\nu(w,v_j)\leq \nu(p,v_j)$. We have to show that $w\in X(K)$. Let $f\in S$ be a polynomial that vanishes on $X$. There is a natural number $l$ such that $f\in K[x_{ij}:\, i\in[l],\,j\in[n]]$. We need to show that $f(w)=0$. Without loss of generality assume that $w_1,\ldots,w_{l'}\in \{v_1,\ldots,v_s\}$ and $w_{l'+1},\ldots,w_{l}\in U(K)$ where $U:=\cup_{i=1}^r V_i$ for some suitable $l'\leq l$. The set $A=\{p_1,p_2,\ldots\}\cap U(K)$ is dense in $U$. Thus $A^{l-l'}$ is dense in $U(K)^{l-l'}$ by \Cref{cor:proddense}. Moreover, since every irreducible component of $U$ has infinitely many $K$-rational points, the open subset $$B=\{(u_{l'+1},\ldots,u_l)\in U(K)^{l-l'}:\, u_i\neq u_j \textrm{ for all }i\neq j\}$$is dense in $U(K)^{l-l'}$ by \Cref{lem:diagonalthin}. \Cref{lem:opendensedense} implies that $B\cap A^{l-l'}$ is dense in $U(K)^{l-l'}$. Note that each element of $B\cap A^{l-l'}$ is a tuple $(p_{i_{l'+1}},\ldots,p_{i_l})$ with $i_{l'+1},\ldots,i_l\in\N$ pairwise distinct and all $p_{i_j}\in U$. Thus since $f$ vanishes on $X$, it vanishes on  $$\{w_1\}\times\cdots\times\{w_{l'}\}\times(B\cap A^{l-l'}).$$ By the above shown density $f$ therefore vanishes on $\{w_1\}\times\cdots\times\{w_{l'}\}\times U(K)^{l-l'}$. In particular $f$ vanishes on $(w_1,\ldots,w_l)$ and thus on $w$.
\end{proof}

\begin{rem}\label{rem:idealoforbit}
 Let $p\in\A^n_{\infty}(K)$ and $X$ its orbit closure. The proofs of \Cref{lem:orbitclossubset} and \Cref{prop:orbitclos} yield an explicit ideal $I\subset S$ whose zero set is $X$. Indeed, consider the decomposition $$V_p=V_1\cup\cdots\cup V_r\subset\A_K^n$$ into irreducible components. Assume that $f_1,\ldots,f_a\in K[x_1,\ldots,x_n]$ generate the vanishing ideal of $V_p$ and for $k=1,\ldots,r$ let $g_{k1},\ldots,g_{km_k}$ generate the vanishing ideal of $\cup_{j\neq k} V_j$. Further, for any polynomial $h\in K[x_1,\ldots,x_n]$ let $h^{i}=h(x_{i1},\ldots,x_{in})\in S$. Then we can choose $I$ to be the ideal generated by all $f_l^i$ and the $G$-orbit of all $$P_{kj}=g^1_{kj}\cdots g^{\nu(p,V_k(K))+1}_{kj}$$for all $k$ with $\nu(p,V_k(K))$ finite (which can only happen if $|V_k(K)|=1$).
\end{rem}

\begin{ex}
 Let $n=2$, $K=\C$ and $$p=((i-1,\textrm{sgn}(i-1))_{i\in\N}\in\A^2_\infty(\C)$$ where $\textrm{sgn}(k)\in\{-1,0,1\}$ is the sign of $k$. Let $X$ be the orbit closure of $p$. The irreducible components of $V_p$ are $\A^1_\C\times\{1\}$ and the singleton $\{(0,0)\}$. Thus the $\C$-rational points $X(\C)$ are all sequences $((a_i,b_i))_{i\in\N}\subset \C^2$ where all but at most one of the $b_i$ are equal to $1$ such that if $b_k\neq 1$, then $a_k=b_k=0$. Further $X$ is the zero set of the ideal of $S$ generated by all $x_{i2}(x_{i2}-1), x_{i1}(x_{i2}-1)$ for $i\in\N$ and $(x_{i2}-1)(x_{j2}-1)$ for $i\neq j$.
\end{ex}

\subsection{Fixed points in $\A^n_\infty$}
Next we want to describe the fixed points of $\A^n_{K,\infty}$ under the action of $G$. We first explain an easy construction. 

\begin{lem}\label{lem:prodpoint}
 Consider a prime ideal $\fp\subset K[x_1,\ldots,x_n]$. Let $\fP\subset S$ be the ideal generated by all $f(x_{i1},\ldots,x_{in})$ for some $f\in\fp$ and $i\in\N$. Then $\fP$ is a prime ideal invariant under the action of $G$.
\end{lem}

\begin{proof}
 It is clear that the ideal $\fP$ is invariant under the action of $G$. It thus remains to show that $\fP$ is a prime ideal. To this end, let $g,h\in S$ such that $g\cdot h\in \fP$. By definition there is a natural number $k\in\N$ such that $g\cdot h$ is the ideal $\fP_k$ of $K[x_{ij}:\, i\in[k],\,j\in[n]]$ that is generated by all $f(x_{i1},\ldots,x_{in})$ with $f\in\fp$ and $i\in[k]$. But the ideals $\fP_k$ are all prime ideals, see for example \cite[Lemma~1.54]{iitaka}. Therefore, one of the factors $g,h$ is in $\fP_k$ and thus in $\fP$.
\end{proof}

\begin{Def}
 For any $y=\fp\in\A^n_K$ we define $y_\infty:=\fP\in\A^n_{K,\infty}$ to be the fixed point constructed from $\fp$ as in \Cref{lem:prodpoint}.
\end{Def}

We observe that $y_\infty$ naturally appears as generic point of a certain orbit closure.

\begin{lem}\label{lem:orbitclosprime}
 Let $p\in\A^n_{K,\infty}(K)$ and assume that $V_p$ is irreducible with generic point $y\in\A^n_K$. Then the orbit closure of $p$ is irreducible with generic point $y_\infty$.
\end{lem}

\begin{proof}
 Let $X$ be the orbit closure of $p$. Then by \Cref{prop:orbitclos} the set $X(K)$ consists of all sequences in $V_p(K)$. Thus by definition the zero set of the prime ideal $y_\infty$ in $\A^n_{K,\infty}(K)$ coincides with $X(K)$. Since $X(K)$ is dense in $X$, the ideal of $X$ is thus $y_\infty=\sqrt{y_\infty}$ by \Cref{lem:langcount}.
\end{proof}

We further observe that if $K\subset L$ is a field extension, then the action of $G$ on $\A^n_{K,\infty}$ induces an action of $G$ on $\A^n_{L,\infty}=\A^n_{K,\infty}\times_K\Spec(L)$ by letting $G$ act on $\Spec(L)$ trivially. The natural morphism of affine schemes $\pi:\A^n_{L,\infty}\to \A^n_{K,\infty}$ is $G$-equivariant. Thus if $x\in \A^n_{L,\infty}$ is a fixed point of $\A^n_{L,\infty}$, then $\pi(x)$ is a fixed point of $\A^n_{K,\infty}$. In particular, for any $y\in\A^n_L$ we obtain the fixed point $\pi(y_\infty)$ of $\A^n_{K,\infty}$.
This gives a convenient way of constructing fixed points in $\A^n_{K,\infty}$. Our next goal is to show that every fixed point of $\A^n_{K,\infty}$ arises in that way.

\begin{lem}\label{lem:runterschneiden}
 Let $x\in\A^n_{K,\infty}$ a fixed point. There is an algebraically closed field extension $K\subset L$ and a closed point $y\in\A^n_{L,\infty}$ such that $\pi(\sigma(y))=x$ for all $\sigma\in G$ where $\pi:\A^n_{L,\infty}\to\A^n_{K,\infty}$ is the natural projection.
\end{lem}

\begin{proof}
 Since $\pi$ is $G$-equivariant, it suffices to find $L$ and a closed point $y\in\A^n_{L,\infty}$ such that $\pi(y)=x$. Let $\kappa(x)$ be the residue field of $x$ and $\kappa(x)\to L$ any homomorphism to an algebraically closed field $L$. This gives a point $x'\in \A^n_{K,\infty}(L)$ which is mapped to $x$ by the map $\A^n_{K,\infty}(L)\to\A^n_{K,\infty}$. Now the claim follows from the fact that the diagram \[\begin{tikzcd}
    \A^n_{L,\infty}(L)\arrow{r} \arrow{d} & \A^n_{L,\infty}\arrow{d} \\
   \A^n_{K,\infty}(L) \arrow{r} & \A^n_{K,\infty} 
  \end{tikzcd}\]  commutes. Note that the left map is bijective and the upper map has its image in the set of closed points of $\A^n_{L,\infty}$.
\end{proof}

\Cref{lem:runterschneiden} enables us to apply \Cref{prop:orbitclos} which can be used to describe the orbit closure of the closed point $y$.

\begin{thm}\label{thm:symprime}
 Let $x\in\A^n_{K,\infty}$ be a fixed point. There is an algebraically closed field extension $K\subset L$ and a point $y\in\A^n_L$ such that $x=\pi(y_\infty)$ where $\pi:\A^n_{L,\infty}\to \A^n_{K,\infty}$ is the natural projection.
\end{thm}

\begin{proof}
 Let $S_L=S\otimes_K L$. The fixed point $x$ is a prime ideal $\fp$ of $S$ and by \Cref{lem:runterschneiden} there is a maximal ideal $\fm$ of $S_L$ with $\fp=S\cap \bigcap_{\sigma\in G} \sigma(\fm)$. The ideal $J=\bigcap_{\sigma\in G} \sigma(\fm)$ is the vanishing ideal of the orbit closure $Z$ of the closed point $\fm$. Thus \Cref{prop:orbitclos} shows that there are integral subschemes $V_1,\ldots,V_r\subset \A_L^n$, closed points $v_1,\ldots,v_s\in \A_L^n$ and natural numbers $m_1,\ldots,m_s$ such that the set $Z(L)\subset\A^n_{L,\infty}(L)$ consists of all sequences in $$V(L)=V_1(L)\cup \cdots\cup V_r(L)\cup\{v_1\}\cup\cdots\cup\{v_s\}$$ where each $v_i$ appears at most $m_i$ times. We are done if $r=1$ and $s=0$. Indeed, in that case we can choose $y\in\A^n_L$ to be the generic point of $V_1$ and $J$ is the prime ideal $y_\infty$ by \Cref{lem:orbitclosprime}.
 
 If $s>0$, then let $\tilde{Z}$ be the set of all sequences in $V$ where $v_i$ appears at most $m_i$ times for $i=1,\ldots,s-1$ and $v_s$ appears at most $m_s-1$ times. Furthermore, let $\tilde{J}\subset S_L$ be the vanishing ideal of $\tilde{Z}$. Since $\tilde{Z}\subset Z(L)$ we have $J\subset\tilde{J}$ and thus $\fp\subset\tilde{J}\cap S$. We claim that $\fp=\tilde{J}\cap S$. Indeed, let $f\in\tilde{J}\cap S$. There is a natural number $k\geq m_s$ such that $f\in K[x_{ij}:\, i\in[k],\,j\in[n]]$. Now from any sequence which is an element of $Z(L)$ we can drop one element so that it becomes a sequence in $\tilde{Z}$. This implies that the product $$\prod_{\sigma\in\Sym(k+1)}f(\underline{x}_{\sigma(1)},\ldots,\underline{x}_{\sigma(k)}),\,\textrm{ where } \underline{x}_{i}:=(x_{i1},\ldots,x_{in}),$$ vanishes on $Z(L)$ and thus lies in $\fp$. Since $\fp$ is prime, one of the factors must lie in $\fp$. Since $\fp$ is invariant, also $f$ lies in $\fp$. Thus we can replace $J$ by $\tilde{J}$. By iterating this process, we can arrive at $s=0$.
 
 We make a similar argument for reducing to the case $r=1$. If $r>1$, we consider the sets $Z_l$ of sequences in $V_l(L)$ for $l=1,\ldots,r$. Since $s=0$, the set $Z$ itself consists of all sequences in $V(L)$. Let $I_l$ be the vanishing ideal of $Z_l$ in $S_L$ for all $l=1,\ldots,r$. We show that $\fp=I_l\cap S$ for a suitable $l$. Assume for the sake of a contradiction that for each $l=1,\ldots,r$ there are $f_l\in I_l\cap S$ with $f_l\not\in\fp$. There is a natural number $k$ such that $f_1,\ldots,f_r\in K[x_{ij}:\, i\in[k],\,j\in[n]]$. By the pigeonhole principle every set of $r\cdot k$ elements from $V(L)$ has a subset of $k$ elements from one the $V_l$. Thus the product $$\prod_{l=1}^r\prod_{\sigma\in\Sym(r\cdot k)} f_l(\underline{x}_{\sigma(1)},\ldots,\underline{x}_{\sigma(k)}),\,\textrm{ where } \underline{x}_{i}:=(x_{i1},\ldots,x_{in}),$$ vanishes on $Z(L)$ and thus lies in $\fp$. Again since $\fp$ is prime, one of the factors must lie in $\fp$ and since $\fp$ is invariant one of the $f_l$ lies in $\fp$. This is a contradiction to our assumption and therefore shows that $\fp=I_l\cap S$ for a suitable $l$. Thus we can replace $I$ by $I_l$ which completes the proof.
\end{proof}

\begin{ex}\label{ex:n1}
 We characterize the fixed points in $\A^1_{K,\infty}$. To this end let $K\subset L$ an algebraically closed field extension and $y\in\A^1_L$. There are the following possibilities:
 \begin{enumerate}
  \item If $y=(0)$, then $\pi(y_\infty)=(0)$ is the generic point of $\A^1_{K,\infty}$.
  \item If $y=(x-\alpha)$ for some $\alpha\in K$, then $\pi(y_\infty)=(x_i-\alpha:\,i\in\N)$ is the constant sequence $\overline{\alpha}=(\alpha)_{i\in \N}$.
  \item If $y=(x-\alpha)$ for $\alpha$ transcendental over $K$, then $\pi(y_\infty)=(x_i-x_j:\,i,j\in\N)$ is the closure $\Delta$ of the set of all constant sequences.
 \end{enumerate}
 Thus every maximal specialization chain of fixed points in $\A^1_{K,\infty}$ is of the form $$(0)\leadsto\Delta\leadsto\overline{\alpha}$$for some $\alpha\in K$.
\end{ex}

We now give a more geometric way of constructing fixed points in $\A_K^n$ which will also be useful for describing arbitrary orbit closures in the next section.

\begin{Def}
Let $X$ be a projective scheme over $K$ and let $\Sigma$ be a closed subscheme of $X\times\A^n_{K}$. For every $x\in X(K)$ the fiber $\pr_2(\pr_1^{-1}(x))$ is a closed subset of $\A_K^n(K)$. We further let $$\Seq(\Sigma)=\{(p_i)_{i\in\N}\subset\A_K^n(K):\,\exists x\in X(K):\forall i\in\N:\, p_i\in\pr_2(\pr_1^{-1}(x))\}.$$ This is a subset of $\A^n_{K,\infty}(K)$. Thus $\Seq(\Sigma)$ consists of all sequences in $\A_K^n(K)$ that are entirely contained in the set of $K$-rational points of one fiber of $\pr_1$. If $\Seq(\Sigma)$ is irreducible, then we denote by $\sq(\Sigma)\in\A^n_{K,\infty}$ the generic point of its closure.
\end{Def}

In the following, we will give a criterion on $\Sigma$ for $\Seq(\Sigma)$ to be irreducible.

\begin{lem}\label{lem:finitesubseq}
 Let $X$ be a projective scheme over $K$ and let $\Sigma$ be a closed subscheme of $X\times\A^n_{K}$. A sequence $(p_i)_{i\in\N}\subset\A_K^n(K)$ is in $\textnormal{Seq}(\Sigma)$ if and only if for every finite set $M\subset\N$ there exists $x\in X(K)$ such that $\{p_i:\, i\in M\}\subset \pr_2(\pr_1^{-1}(x))$.
\end{lem}

\begin{proof}
 One direction of the claim is clear. Thus let $(p_i)_{i\in\N}\subset\A_K^n(K)$ not in $\textnormal{Seq}(\Sigma)$. This means that $\cap_{i\in\N}\textnormal{pr}_1(\textnormal{pr}_2^{-1}(p_i))=\emptyset$. But since $X(K)$ is quasi-compact, there is a finite set $M\subset\N$ such that $\cap_{i\in M}\textnormal{pr}_1(\textnormal{pr}_2^{-1}(p_i))=\emptyset$.
\end{proof}

\begin{lem}\label{prop:seqclosed}
 Let $X$ be a projective scheme over $K$, let $\Sigma$ be a closed subscheme of $X\times\A^n_{K}$ and let $I$ be the vanishing ideal of $\Seq(\Sigma)$. Let $k\in\N$ and $$I_k=I\cap K[x_{ij}:\,i\in[k],j\in[n]].$$ Furthermore, consider the $k$-fold fiber product $$\Sigma_k=\Sigma\times_{X}\cdots\times_{X}\Sigma\subset X\times(\A^n_K)^k $$ of $\Sigma$ over $X$. Then:
 \begin{enumerate}
     \item $I_k$ is the vanishing ideal of $\pr_2(\Sigma_k)\subset(\A^n_K)^k$.
     \item The zero set of $I_k$ is $\pr_2(\Sigma_k)\subset(\A^n_K)^k$.
     \item The zero set of $I$ in $\A^n_{K,\infty}(K)$ is $\Seq(\Sigma)$.
     \item $I$ is invariant under $G$.
     \item If $\Sigma_k$ is irreducible for all $k$, then $I$ is a prime ideal.
 \end{enumerate}
\end{lem}

\begin{proof}
 By the fundamental theorem of elimination theory $Y_k=\pr_2(\Sigma_k)$ is a closed subscheme of $(\A^n_K)^k$. Thus the first statement implies the second statement. The $K$-rational points are dense in $Y_k$ and thus it for proving (1) suffices to show that $I_k$ equals the vanishing ideal $J_k$ of $$Y_k(K)=\{(p_1,\ldots,p_k)\in(\A^n_K)^k(K):\, \exists x\in \pp^m_K(K):\, p_i\in \pr_2(\pr_1^{-1}(x)) \textnormal{ for }i\in[k] \}.$$ Let $f\in I_k$ and $(p_1,\ldots,p_k)\in Y_k(K)$. Then we have $(p_1,\ldots,p_k,p_k,\ldots)\in\Seq(\Sigma)$ and therefore $f(p_1,\ldots,p_k)=0$ since $f\in I$. This shows $I_k\subset J_k$. The other inclusion $J_k\subset I_k$ is clear. Part (3) follows by combining (2) with \Cref{lem:finitesubseq}. Part (4) is clear and part (5) follows from (1) and the fact that $I$ is a prime ideal if all $I_k$ are prime ideals.
\end{proof}

\begin{lem}\label{lem:flatirred}
 Let $f:X\to Y$ be a flat morphism of schemes of finite type over $K$ with $Y$ irreducible. If there is a nonempty open subset $U\subset Y(K)$ such that for all $y\in U$ the fiber $f^{-1}(y)$ is irreducible, then $X$ is irreducible.
\end{lem}

\begin{proof}
 Let $X_1,\ldots,X_r$ with $r>1$ be the irreducible components of $X$. It follows for instance from \cite[Prop.~III.9.5]{Hart77} that the generic point of each $X_i$ is mapped to the generic point of $Y$. Thus restricting $f$ to any $X_i$ gives a dominant morphism $X_i\to Y$. Thus by Chevalley's theorem the image of $X_i(K)\setminus\cup_{j\neq i}X_j(K)$ under $f$ contains a nonempty open subset $U_i$ of $Y(K)$. Since $Y$ is irreducible, the intersection of all $U_i$ and $U$ is nonempty. Let $y$ be an element of this intersection. Then $f^{-1}(y)$ is not contained in any of the $X_i(K)$ and thus is not irreducible. This contradicts our assumption on $U$.
\end{proof}

\begin{thm}\label{thm:familiesfixed}
 Let $X$ be an irreducible projective scheme over $K$ and let $\Sigma$ be a closed subscheme of $X\times\A^n_{K}$ such that $\pr_1:\Sigma\to X$ is flat. If there is a nonempty open subset $U\subset X(K)$ such that for all $x\in U$ the fiber $\pr_1^{-1}(x)$ is irreducible, then $\Seq(\Sigma)$ is irreducible. In particular, we obtain a fixed point $\sq(\Sigma)\in\A^n_{K,\infty}$.
\end{thm}

\begin{proof}
 By \Cref{prop:seqclosed} it suffices to show that $\Sigma_k$ is irreducible for all $k$. Since flatness is preserved under composition and base change \cite[Prop.~III.9.2]{Hart77}, the projection $\Sigma_k\to X$ is flat. For every $x\in U$ the fiber of this projection is the direct product of irreducibles  and thus is irreducible itself. Therefore, the claim follows then from \Cref{lem:flatirred}.
\end{proof}

\begin{ex}
 Let $\Sigma\subset\G(r,n)\times\A^n_K$ be the closed subscheme whose $K$-rational points are of the form $(H,x)$ with $x\in H$. Each fiber of $\pr_1$ is a linear subspace of dimension $r$. Thus we can apply \Cref{thm:familiesfixed} to deduce that $\Seq(\Sigma)$ is irreducible. The set $\Seq(\Sigma)$ consists of all sequences in $\A^n_K(K)$ that are entirely contained in some linear subspace of dimension $r$. Therefore, the fixed point $\sq(\Sigma)$ is the prime ideal of $S$ generated by all $(r+1)\times (r+1)$ minors of the matrix $$\begin{pmatrix}
 x_{11} & x_{21} & \ldots & x_{k1} & \ldots \\
 \vdots & \vdots & & \vdots & \\
 x_{1n} & x_{2n} & \cdots & x_{kn} & \ldots 
 \end{pmatrix}$$ with $n$ rows and infinitely many columns.
\end{ex}

\begin{ex}\label{ex:infichain}
 Contrary to \Cref{ex:n1}, in the case $n>1$ there are infinite specialization chains of fixed points in $\A^n_{K,\infty}$. Indeed, let $$\Sigma_d\subset\pp(K[x_1,\ldots,x_n]_{\leq d})_K\times\A^n_K$$ the generic hypersurface of degree (at most) $d$, i.e. the $K$-rational points of $\Sigma_d$ are of the form $([f],x)$ where $f\in K[x_1,\ldots,x_n]$ is a nonzero polynomial of degree at most $d$ and $x\in\A^n_K(K)$ such that $f(x)=0$. A general fiber of $\pr_1$ is an irreducible hypersurface of degree $d$. It further follows from the fact that open immersions are flat together with \cite[Prop.~III.9.9]{Hart77} that $\pr_1$ is flat. Thus by \Cref{thm:familiesfixed} the set $\Seq(\Sigma_d)$ is irreducible. It consists of those sequences $(p_i)_{i\in\N}\subset \A_K^n(K)$ for which there exists a non-zero polynomial of degree at most $d$ that vanishes on all $p_i$. In particular, we have $\Seq(\Sigma_d)\subset \Seq(\Sigma_{d+1})$. Thus letting $z_d=\sq(\Sigma_d)$ we have the following infinite specialization chain of fixed points in $\A^n_{K,\infty}$: $$\cdots\leadsto z_4\leadsto z_3\leadsto z_2\leadsto z_1.$$
\end{ex}

\subsection{Orbit closures of arbitrary points in $\A^n_\infty$}

Finally we describe the vanishing ideals of orbit closures of arbitrary points in the infinite affine space.

\begin{lem}\label{lem:cutdownideals}
 Let $\varphi:X\to Y$ a morphism of affine schemes $X=\Spec(B)$ and $Y=\Spec(A)$. Let $\varphi^*:A\to B$ the pullback. Let $I_B\subset B$ an ideal and $I_A=(\varphi^*)^{-1}(I_B)$.
 \begin{enumerate}
     \item  Let $S\subset X$ the zero set of $I_B$. The zero set of $I_A$ in $Y$ is the closure of $\varphi(S)$.
     \item Assume further that $\varphi:X\to Y$ a morphism of affine schemes of countable type over $K$ and let $S\subset X(K)$ the zero set of $I_B$. Then the zero set of $I_A$ in $Y(K)$ is the closure of $\varphi(S)$.
 \end{enumerate}
\end{lem}

\begin{proof}
 (1) is a basic fact from algebraic geometry and (2) can be proved analogously. To that end let $T\subset Y(K)$ the zero set of $I_A$. The containment $\varphi(S)\subset T$ is then clear. For proving the other inclusion let $f\in A$ vanish on $\varphi(S)$ which implies that $\varphi^*(f)$ vanishes on $S$. By \Cref{lem:langcount} this implies that $\varphi^*(f)^k\in I_B$. But that shows that $f\in\sqrt{I_A}$ and thus $f$ vanishes on $T$.
\end{proof}

\begin{thm}\label{thm:orbitclos}
 Let $y\in\A^n_{K,\infty}$ and $I$ the vanishing ideal of the orbit closure of $y$. There is an affine integral scheme $X$ of finite type over $K$, closed subschemes $\Sigma_0,\ldots,\Sigma_r\subset X\times\A^n_K$ and natural numbers $m_1,\ldots, m_r$ such that the following holds. The ideal $I$ is the vanishing ideal of all sequences $(p_i)_{i\in\N}\subset\A_K^n(K)$ for which there exists $x\in X(K)$ with:
 \begin{enumerate}
  \item For all $i\in\N$ we have $p_i\in\Sigma(0,x)$, and
  \item for all $j\in[r]$ there are at most $m_j$ indices $i\in\N$ such that $p_i\not\in\Sigma(j,x)$.
 \end{enumerate}
 Here $\Sigma(j,x)\subset\A^n_K(K)$ denotes the fiber of $x\in X(K)$ under $\pr_1:\Sigma_j\to X$.
\end{thm}

\begin{proof}
 As in the proof of \Cref{lem:orbitclosprime} there is an algebraically closed field extension $K\subset L$ and a point $x\in\A^n_{L,\infty}(L)$ with $\pi(x)=y$ where $\pi:\A^n_{L,\infty}\to\A^n_{K,\infty}$ is the natural projection. Letting $Y\subset\A^n_{K,\infty}$ be the orbit closure of $y$ and $X\subset\A^n_{K,\infty}$ the orbit closure of $x$, we have $Y=\overline{\pi(X)}$ since $\pi$ is $G$-equivariant. As in \Cref{rem:idealoforbit} we can find explicit generators of an ideal $I_L\subset S_L:=S\otimes_K L$ whose zero set is $X$. In order to describe these here, we use the notation $h^i=h(x_{i1},\ldots, x_{in})\in S_L$ for $h\in L[x_1,\ldots,x_n]$ as in \Cref{rem:idealoforbit}. There are $m_1,\ldots,m_r\in\N$ and for $k=0,\ldots,r$ polynomials $g_{k1},\ldots,g_{ka_k}\in L[x_1,\ldots,x_n]$ such that we can choose $I_L$ to be generated by $g_{01}^i,\ldots,g_{0a_0}^i$ for $i\in\N$ and the $G$-orbits of the polynomials $p_{kj}=g_{kj}^1\cdots g_{kj}^{m_k+1}$ for $k=1,\ldots,r$ and $j=1,\ldots,a_k$. By \Cref{lem:cutdownideals} $Y$ is the zero set of $I_K=I_L\cap S$. By \cite{cohen} there are finitely many $f_1,\ldots,f_s\in S$ whose $G$-orbits generate $I_K$. Since $I_K=S\cap I_L$, each of these $f_l$ can be written as a linear combination of the $g_{0i}$ and the $p_{kj}$ with coefficients $h\in S_L$. We let $A$ be the $K$-algebra that is generated by the (finitely many) coefficients of these $h$ and all $g_{kj}$ as polynomials over $L$. Thus the $g_{kl}$ lie in the subalgebra $S_A=S\otimes_K A$ of $S_L$. We let $I_A$ be the ideal in $S_A$ generated by $g_{01}^i,\ldots,g_{0a_0}^i$ for $i\in\N$ and the $G$-orbits of the polynomials $p_{kj}$ for $k=1,\ldots,r$ and $j=1,\ldots,a_k$. By construction we have $I_K\subset I_A\subset I_L$. This shows that $I_K=I_A\cap S$. By \Cref{lem:cutdownideals} a dense subset of $Y(K)$ is thus given by the image of the $K$-points of the zero set of $I_A$ under the projection $\Spec(S_A)\to\A^n_{K,\infty}$. Therefore, the claim follows by choosing $X=\Spec(A)$ and $\Sigma_k$ to be the zero set of the $g_{k1},\ldots,g_{ka_k}$ in $X\times\A_K^n$.
\end{proof}

\begin{ex}
 Let $n=1$ and consider the ideal of $S$ generated by $x_3,x_4,x_5,\ldots$. The zero set of this ideal is clearly irreducible. Denote by $y\in\A^1_{K,\infty}$ its generic point. The $K$-rational points of the orbit closure of $y$ is the set of all sequences with at most two nonzero members. Its vanishing ideal is generated by the $G$-orbit of $x_1x_2x_3$. In the notation of \Cref{thm:orbitclos} we can choose $X=\Spec(K)$, $\Sigma_0=\A^1_K$, $\Sigma_1=\{0\}$ and $m_1=2$.
\end{ex}

We close with a conjecture on the structure of the orbit space $\A^n_{K,\infty}/G$. This space is rather badly behaved as it does not even satisfy the $T_0$ condition. Indeed, for example the sequences $(2i)_{i\in\N}$ and $(2i+1)_{i\in\N}$ of even and odd numbers have distinct orbits in $\A^1_{\C,\infty}$ but their orbit closures agree. In order to avoid such pathologies we pass to the Kolmogorov quotient $\textnormal{KQ}(\A^n_{K,\infty}/G)$, i.e., we identify topologically indistinguishable points. By a classical result first proved by Cohen in \cite{cohen} this space is noetherian. We conjecture that it is even a spectral space.

\begin{con}\label{con:sober}
 The Kolmogorov quotient $\textnormal{KQ}(\A^n_{K,\infty}/G)$ is a spectral space.
\end{con}

Recall that a noetherian $T_0$ space is a \emph{spectral space} if it is \emph{sober} in the sense that every irreducible subset has a generic point. A result of Hochster \cite{hochster} says that every spectral space can be realized as the spectrum of a commutative ring. Our interest in \Cref{con:sober} stems from the following fact.

\begin{lem}
 If \Cref{con:sober} is true, then every $G$-invariant closed subset $X$ of $\A^n_{K,\infty}$ is the union of the orbit closures of finitely many points. In particular, the vanishing ideal of $X$ is the intersection of finitely many ideals $I$ as in \Cref{thm:orbitclos}.
\end{lem}

\begin{proof}
 The set $X$ maps to a closed subset of $\textnormal{KQ}(\A^n_{K,\infty}/G)$. Since this latter space is noetherian, this closed subset is the union of its finitely many irreducible components. If \Cref{con:sober} is true, then each of these irreducible components has a generic point. This shows that $X$ is the union of finitely many orbit closures.
\end{proof}

The case $n=1$ of \Cref{con:sober} follows from the main result of \cite{nagpal2020symmetric}. Indeed, a description of the generic point of an irreducible subset is given in the last paragraph of \cite[\S 1.1]{nagpal2020symmetric}. Note that the conjectural description of vanishing ideals is reminiscent of the Shift Theorem \cite[Theorem~3.1.1]{draisma2021components} but without the need of paasing to a localization.

\section{Equivariant semi-algebraic geometry}\label{sec:semialg}
Let $A$ be a finitely generated, reduced $\R$-algebra and let $V=\Spec(A)$. Further let $B=A[x_1,x_2,x_3,\ldots]$ with the natural $G$ action on it and let $X=\Spec(B)$. We equip the space $X(\R)$ with the topology generated by the open sets $U(f)=\{p\in X(\R):\,f(p)>0\}$ for $f\in B$. We have $X(\R)=V(\R)\times\A^1_{\R,\infty}(\R)$ where the projection $\pi:X\to V$ is induced by the inclusion $A\hookrightarrow B$. 
A fundamental result in real algebraic geometry is the so called quantifier elimination due to Tarski and Seidenberg which in particular guarantees that a projection of a semi-algebraic set is again semi-algebraic.  
The goal of this section is to prove a version of this theorem (\Cref{thm:tarskiseid}) for the projection $\pi$ in the case of basic equivariant semi-algebraic sets.

\begin{Def}\label{def:semialgebaric}
 A \emph{basic equivariant semi-algebraic set} is a subset of $X(\R)$ of the form \[\bigcap_{\sigma\in G}\{p\in X(\R):\, (\sigma f_1)(p)\geq0,\ldots, (\sigma f_r)(p)\geq0, \, (\sigma g_1)(p)>0,\ldots, (\sigma g_s)(p)>0\}\]for some $f_1,\ldots,f_r,g_1,\ldots,g_s\in B$. 
\end{Def}
\begin{rem}
Similarly  to the finite variable case one might define \emph{equivariant semi-algebraic set} as the sets obtained by a finite boolean combination of basic equivariant semi-algebraic sets. However, the quantifier elimination which we show below is currently only established for the basic equivariant semi-algebraic case.
\end{rem}

In order to show quantifier elimination for these equivariant basic semi-algebraic sets we will rely on the following property of  symmetric  semi-algebraic sets.
\begin{lem}\label{lem:salem} Let $p\in\R\cup\{\infty\}$.
 A symmetric semi-algebraic set $S\subseteq\R^n$ has the following property:
 $$\exists c<p:\forall x_1\in[c,p):\exists c_{x_1}<p:\forall x_2\in [c_{x_1},p): \cdots $$ $$\cdots\exists c_{x_1\ldots x_{n-1}}<p:\forall x_n\in [c_{x_1\ldots x_{n-1}},p) \in\R:\,(x_1,\ldots,x_n)\in S$$
 if and only if there is a strictly monotonously increasing sequence $(p_i)_{i\in\N}$ of real numbers with $p=\lim_{i\to\infty}p_i$ such that for every pairwise distinct natural numbers $j_1,\ldots,j_n$ we have $(p_{j_1},\ldots,p_{j_n})\in S$.
\end{lem}

\begin{proof}
 Suppose that $S$ has the above property. We define our sequence $(p_i)_{i\in\N}$ recursively. Let $p_1=c$ and \small$$p_i=\max(\{c_{p_{k_1},\ldots,p_{k_j}}:\,j=1,\ldots,\min(n,i)-1,\, 1\leq k_1<\cdots<k_j<i\}\cup\{p_1,\ldots,p_{i-1}\})$$\normalsize for $i>1$. Clearly, the sequence is monotonously increasing. After replacing $p_i$ by $\max(p_i,p-1/i)$ if $p\in\R$ and by $\max(p_i,n)$ if $p=\infty$ we get $p=\lim_{i\to\infty}p_i$. Furthermore, after passing to a suitable subsequence, it is strictly monotonously increasing. Finally, we have by construction $(p_{j_1},\ldots,p_{j_n})\in S$ for all natural numbers $j_1<\cdots<j_n$ so the rest of the claim follows from the symmetry of $S$.
 
 Conversely, let $(p_i)_{i\in\N}$ be a sequence with the desired properties. We prove the statement by induction on $n$. In the case $n=1$ the set $S\subseteq\R$ is a finite union of intervals since semi-algebraic. Since it contains $(p_i)_{i\in\N}$ which converges to $p$ we thus find a $c<p$ such that $S$ contains $[c,p)$. For the induction step we consider the set $$S'=\{x_1<p:\exists c_{x_1}<p:\forall x_2\in [c_{x_1},p): \cdots$$ $$\cdots \exists c_{x_1\ldots x_{n-1}}<p:\forall x_n\in [c_{x_1\ldots x_{n-1}},p) \in\R:\,(x_1,\ldots,x_n)\in S\}.$$ We obtain $p_k\in S$ for all $k\in\N$ from applying the induction hypothesis to the set $\{(x_2,\ldots,x_n)\in\R^{n-1}:(p_k,x_2,\ldots,x_n)\in S\}$ and the sequence $(p_{k+i})_{i\in\N}$. This shows in particular that $S'$ has $p$ in its closure. Therefore, since $S'$ is also semi-algebraic, we find as above $c\in\R$ such that $S'$ contains all $x_1\in [c,p)$. This proves the claim.
\end{proof}
The results from \Cref{lem:salem} now allow us to establish the following equivariant quantifier elimination.
\begin{thm}\label{thm:tarskiseid}
 Let $Y\subseteq X(\R)$ be a basic equivariant semi-algebraic set. Then the image $\pi(Y)$ is a semi-algebraic subset of $V(\R)$.
\end{thm}

\begin{proof}
 We have to show that the condition on the fiber over a point $v\in V(\R)$ being nonempty is a semi-algebraic condition on $v$. To that end let $f_1,\ldots,f_r,g_1,\ldots,g_s\in \R[x_1,x_2,\ldots]$
 and consider \small\[T=\bigcap_{\sigma\in G}\{p=(p_i)_{\in\N}:\, (\sigma f_1)(p)\geq0,\ldots, (\sigma f_r)(p)\geq0, \, (\sigma g_1)(p)>0,\ldots, (\sigma g_s)(p)>0\}.\]\normalsize Then we already have $f_1,\ldots,f_r,g_1,\ldots,g_s\in \R[x_1,\ldots,x_n]$ for some $n\in\N$. Let \[S=\bigcap_{\sigma\in \Sym(n)}\{p\in \R^n: (\sigma f_1)(p)\geq0,\ldots, (\sigma f_r)(p)\geq0, \, (\sigma g_1)(p)>0,\ldots, (\sigma g_s)(p)>0\}.\] A sequence $(p_i)_{\in\N}$ of real numbers lies in $T$ if and only if for every pairwise distinct natural numbers $j_1,\ldots,j_n$ we have $(p_{j_1},\ldots,p_{j_n})\in S$. Note that any subsequence of a sequence in $T$ will also lie in $T$. It follows that $T$ is nonempty if and only if one of the following conditions on $T$ is true:\begin{enumerate}[a)]
 \item $T$ contains a strictly monotonously increasing sequence;
 \item $T$ contains a strictly monotonously decreasing sequence;
 \item $T$ contains a constant sequence.
 \end{enumerate}
 The statement now follows, since every of these conditions is semi-algebraic. 
 Indeed, by \Cref{lem:salem} condition $a)$ is semi-algebraic. Analogously, condition $b)$ is semi-algebraic as well. Finally, let $\tilde{f}_i=f_i(x,\ldots,x)$ and $\tilde{g}_j=g_j(x,\ldots,x)$ for all $i,j$. Then condition $c)$ is equivalent to $$\emptyset\neq\{p\in \R: \tilde{f}_1(p)\geq0,\ldots, \tilde{f}_r(p)\geq0, \, \tilde{g}_1(p)>0,\ldots, \tilde{g}_s(p)>0\}$$ which is also semi-algebraic.
\end{proof}

The argument from the proof of \Cref{thm:tarskiseid} does not extend to arbitrary $n$ as the next example shows.

\begin{ex}\label{ex:last}
 Consider the set $S$ of sequences $(p_i)_{i\in\N}\subset\R^5$ that satisfy $$p_i=(a_i,b_i,x_i,y_i,z_i),\, a_i=a_j, b_i=b_j, a_i^2+b_i^2>1 \textrm{ and }a_ix_i+b_iy_i+z_i>0$$for all $i,j\in\N$. Let $S'$ be the set obtained by projecting $S$ onto the last three variables. Thus $S'$ consists of those sequences $s=(x_i,y_i,z_i)_{i\in\N}$ with the property that the convex set $$K(s)=\{(a,b)\in\R^2:\,ax_i+by_i+z_i>0\textrm{ for all }i\in\N\}$$intersects the complement of the unit disc. We will show that $S'$ can not be characterized by bounded-size sub-sequences. To that end we will construct for every $N\in\N$ a sequence of $N$ points in $\R^3$ that can not be extended to a sequence in $S'$ while every sub-sequence of length $N-1$ can: Let $(x_i,y_i,z_i)\in\R^3$ for $i=1,\ldots,N$ such that $z_i>0$ and such that the linear polynomial $ax_i+by_i+z_i=0$ defines the $i$th edge of the regular $N$-gon $C_n$ inscribed in the unit circle $a^2+b^2=1$. Clearly, $$C_n=\{(a,b)\in\R^2:\,ax_i+by_i+z_i>0\textrm{ for all }i=1,\ldots,N\}$$ is entirely contained in the closed unit disc but $$\{(a,b)\in\R^2:\,ax_i+by_i+z_i>0\textrm{ for all }i=1,\ldots,N\textrm{ with }i\neq j\}$$is not for all $j=1,\ldots, N$.
\end{ex}
\begin{rem}
In fact, \Cref{ex:last} actually shows that, similarly to the setup studied in \cite{moitra}, it is impossible to expect quantifier elimination for any kind of $G$-equivariant description involving only finitely many orbits of polynomials, whenever $n>1$.
\end{rem}
\bigskip

  \bibliographystyle{abbrv}
  \bibliography{biblio}
 \end{document}